\documentclass{article}
\usepackage{graphicx} 
\usepackage[utf8]{inputenc}
\usepackage[english]{babel}
\usepackage{amsmath}
\usepackage{amsfonts}
\usepackage{amssymb}
\usepackage{hyperref}
\usepackage{amsthm}
\usepackage[left=3cm,right=3cm,top=3cm,bottom=3cm]{geometry}

\theoremstyle{plain}
\newtheorem{teo}{}[section]
\newtheorem{prop}[teo]{Proposition}
\newtheorem{cor}[teo]{Corollary}

\newtheorem{thm}[teo]{Theorem}
\theoremstyle{definition}

\newcommand\blfootnote[1]{%
  \begingroup
  \renewcommand\thefootnote{}\footnote{#1}%
  \addtocounter{footnote}{-1}%
  \endgroup
}

\title{Every group retraction can be realized as a topological retraction}
\author{Pedro J. Chocano}
\date{}

\begin{document}

\maketitle

\begin{abstract}

Given a group retraction \( r: G \rightarrow H \), we construct a finite topological space \( X_r \) of height 1, together with a topological retraction \( \overline{r}: X_r \rightarrow X_r \), such that the group of automorphisms \( \mathrm{Aut}(X_r) \) (or the group of self-homotopy equivalences \( \mathcal{E}(X_r) \)) of $X_r$ is isomorphic to \( G \), and \( \mathrm{Aut}(\overline{r}(X_r)) \) (or \( \mathcal{E}(\overline{r}(X_r)) \)) is isomorphic to \( H \). Moreover, there is a natural map \(\overline{r}' : \mathrm{Aut}(X_r) \rightarrow \mathrm{Aut}(\overline{r}(X_r)) \) that coincides with the original group retraction \( r \). As a direct consequence of this construction, we show that height 1 is the minimal height required to realize any finite group as the group of automorphisms (or the group of self-homotopy equivalences) of a finite topological space, except in the case where \( G \) is a symmetric group. In that unique case, the group can be realized by a finite topological space of height 0.

\end{abstract}

\section{Introduction}\label{sec_intro}
A classical problem in mathematics is whether an algebraic structure can be realized by a topological space. In recent years, this question has been fruitfully explored in the context of non-finite Alexandroff spaces and finite topological spaces (or partially ordered sets). We begin by briefly reviewing some of the most relevant recent results in this direction.

In 1946, G. Birkhoff~\cite{birkhoff1936order} proved that every finite group $G$ can be realized as the group of automorphisms  $\textnormal{Aut}(X)$ of a finite partially ordered set $X$. Decades later, M.C. Thornton~\cite{thornton1972spaces} reduced the number of points required for such realizations, initiating the study of general constructions that realize any group \( G \) using as few points as possible. J.A. Barmak and E.G. Minian~\cite{barmak2009automorphism} improved upon Thornton's results, and more recently, J.A. Barmak~\cite{barmak2020automorphism2} showed that it suffices to use \( 4|G| \) points to realize a group \( G \). The problem of finding the minimal number of points needed to realize cyclic groups was addressed and solved in~\cite{barmak2024smallest} by J.A. Barmak and A.N. Barreto.

The original construction in~\cite{barmak2009automorphism} inspired the question of whether any (not necessarily finite) group can be realized as the group of self-homotopy equivalences \( \mathcal{E}(X) \) of an Alexandroff space \( X \). This problem was affirmatively resolved in~\cite{chocano2020topological}, where the authors generalized previous results to non-finite groups. Building on this, the same authors showed in~\cite{chocano2020some} that every group homomorphism \( f: G \rightarrow H \) can be realized by an Alexandroff space \( X_f \), satisfying \( \mathrm{Aut}(X_f) \simeq G \), \( \mathcal{E}(X_f) \simeq H \), and such that the natural map \( \tau: \mathrm{Aut}(X_f) \rightarrow \mathcal{E}(X_f) \), given by \( \tau(g) = [g] \), coincides with \( f \). Furthermore, in~\cite{chocano2025realizing}, one of the previous authors proved that the regular representation of a finite group can be realized by a finite topological space using its homology groups and the natural actions of \( \mathrm{Aut}(X) \) and \( \mathcal{E}(X) \) on them. A generalization of this result can be found in \cite{viruel2024permutation}, where C. Costoya, R. Gomes and A. Viruel considered the more general framework of realizing some group actions on permutation modules.

In this paper, we focus on the following problem: given a group retraction \( r: G \rightarrow H \) between two finite groups, is it possible to realize it using topological spaces within the topological category (or the homotopical category)? We provide a positive answer to this question through a concrete construction of a finite topological space \( X_r \) of height 1. Note that this problem differs from the one introduced above regarding the realization of group homomorphisms, as the earlier problem considers both the group of automorphisms and the group self-homotopy equivalences simultaneously, while here we restrict ourselves to only one category.

Moreover, the proposed construction of the space \( X_r \) provides a new approach to the following question: given a finite group \( G \), what is the minimal possible height of a finite topological space \( X \) such that \( \mathrm{Aut}(X) \cong G \) (respectively, \( \mathcal{E}(X) \cong G \))? This question was previously addressed in \cite[Theorem 3.1]{viruel2024permutation}, where it was shown that height 1 suffices unless \( G \) is a symmetric group.

This leads naturally to a further problem concerning the simultaneous minimality of height and cardinality: \emph{given a finite group \( G \) that is not symmetric, what is the minimal cardinality of a finite topological space of height 1 realizing \( G \)?} This last problem is not addressed in the present paper. However, it opens an interesting direction for future research, and we hope to explore it in a forthcoming work.

The organization of the paper is as follows. In Section~\ref{sec_preliminaries}, we introduce the basic notation and concepts from the theory of finite topological spaces to make the paper as self-contained as possible. In Section~\ref{sec_main_result}, we present the main results and provide the constructions that prove them.

\section{Preliminaries}\label{sec_preliminaries}

For a comprehensive introduction to the theory of finite topological spaces, we refer the reader to \cite{barmak2011algebraic, may1966finite}.

Let \((X, \tau)\) be a finite \(T_0\)-topological space and let \(x \in X\). The open set \(U_x\) (respectively, the closed set \(F_x\)) is defined as the intersection of all open (respectively, closed) sets containing \(x\). We define a relation \(x \leq_\tau y\) if and only if \(U_x \subseteq U_y\). This relation induces a partial order on \(X\). Conversely, given a finite partially ordered set (or poset) \((X, \leq)\), the collection of upper sets forms a basis for a \(T_0\)-topology on \(X\). These two constructions are mutually inverse. Moreover, a map \(f: (X, \tau) \to (Y, \sigma)\) is continuous if and only if \(f: (X, \leq_\tau) \to (Y, \leq_\sigma)\) is order-preserving. From this, it can be deduced a crucial result (\cite{alexandroff1937diskrete}).
\begin{thm}
    The category of finite \(T_0\)-topological spaces is isomorphic to the category of finite partially ordered sets. 
\end{thm}
This equivalence allows us to treat finite \(T_0\)-topological spaces and finite posets interchangeably. From now on, we will not distinguish between them. The \emph{Hasse diagram} of a finite poset \(X\) is a directed graph whose vertices are the elements of \(X\), with a directed edge from \(x\) to \(y\) whenever \(x < y\) and there is no \(z \in X\) such that \(x < z < y\). In subsequent diagrams, we omit the orientation of edges and assume an upward orientation.

We now introduce some basic notions from this theory. Given a finite poset \(X\) and \(x \in X\), the \emph{height} of \(X\), denoted \(\textnormal{ht}(X)\), is one less than the maximum number of elements in a chain of \(X\). The height of a point \(x\), denoted \(\textnormal{ht}(x)\), is defined as \(\textnormal{ht}(U_x)\). A \emph{fence} in \(X\) is a finite sequence \(\gamma = \{x_1, \ldots, x_n\}\) such that any two consecutive elements are comparable.

\begin{prop}\label{prop_fences}
    Let $X$ be a finite \(T_0\)-topological space and let $f:X\rightarrow X$ be a homeomorphism. If $\gamma=\{x_1,...,x_n\}$ is a fence, then $f(\gamma)=\{f(x_1),...,f(x_n) \}$ is a fence, $|F_{x_i}|=|F_{f(x_i)}|$, $|U_{x_i}|=|U_{f(x_i)}|$ and $\textnormal{ht}(x_i)=\textnormal{ht}(f(x_i))$ for every $i=1,...,n$.
\end{prop}

Let \(X\) be a finite \(T_0\)-topological space and \(x \in X\). We say that \(x\) is a \emph{down beat point} (respectively, \emph{up beat point}) if \(U_x \setminus \{x\}\) has a maximum (respectively, \(F_x \setminus \{x\}\) has a minimum). Removing beat points does not change the homotopy type of \(X\); see \cite{stong1966finite}. A finite space \(X\) is called \emph{minimal} if it contains no beat points.

We state a key result that will be used to prove one of our main theorems.

\begin{thm}\label{thm_aut_igual_E}
    Let $X$ be a finite \(T_0\)-topological space. If $X$ is a minimal finite space, then $\textnormal{Aut(X)}=\mathcal{E}(X)$.
\end{thm}



To conclude this section, we recall some basic categorical concepts. Let \( \mathcal{C} \) be a category. An object \( A \) is called a \emph{retract} of an object \( B \) if there exist morphisms \( r: B \rightarrow A \) and \( i: A \rightarrow B \) such that \( r \circ i = \mathrm{id}_A \). The morphism \( r \) is referred to as a \emph{retraction} of \( B \) onto \( A \). In this paper, we consider this notion in the categories of topological spaces \( \mathbf{Top} \), the homotopy category of topological spaces \( \mathbf{HTop} \), and the category of groups \( \mathbf{Grp} \).

It is straightforward to verify that given a retraction \( r: G \rightarrow H \) in \( \mathbf{Grp} \), one has
\[
G = \ker(r) \cdot i(H), \quad \text{and} \quad \ker(r) \cap i(H) = \{e\},
\]
where  \( e \) is the identity element.

Let \( X \) and \( Y \) be topological spaces, and let \( \overline{r}: X \rightarrow Y \) be a retraction in \( \mathbf{Top} \) (or \( \mathbf{HTop} \)), with \( \overline{i}: Y \rightarrow X \) a morphism such that \( \overline{r} \circ \overline{i} = \mathrm{id}_Y \). Then, there exists a natural map
\[
\overline{r}' : \mathrm{Aut}(X) \rightarrow \mathrm{Aut}(Y) \quad \text{(or } \overline{r}' : \mathcal{E}(X) \rightarrow \mathcal{E}(Y) \text{)}
\]
defined by
\[
\overline{r}'(f) = \overline{r} \circ f \circ \overline{i},
\]
for every \( f \in \mathrm{Aut}(X) \) (or \( f \in \mathcal{E}(X) \)), provided that the restriction \( f|_{\overline{i}(Y)} \in \mathrm{Aut}(\overline{i}(Y)) \) (or \( f|_{\overline{i}(Y)} \in \mathcal{E}(\overline{i}(Y)) \)).

We say that a group retraction \( r: G \rightarrow H \) is \emph{realizable} in the category \( \mathbf{Top} \) or \( \mathbf{HTop} \) if there exists a topological retraction \( \overline{r}: X \rightarrow Y \) satisfying the above conditions and such that \( r = \overline{r}' \).
\section{Main results}\label{sec_main_result}

Let $r:G\rightarrow H$ be a retraction of groups, where $G$ is a finite group of order $n$. Let $S_1$ and $S_2$ be generating systems for $\ker(r)$ and $H$, respectively. Then  \( S =S_1\cup S_2= \{s_1, \ldots, s_m\} \) is a generating system for $G$. For each \( g \in G \), we define a finite set \( C_g \) consisting of the following elements: \begin{align*}   
 C_g=\{& (g,-1),(g,0),(g,*_0) ,(g,1),(g,*_1),(g,*_2),(g,2),(g,*_3),(g,3),...,(g,*_m),(g,m),\\ 
 &(g,1,0),(g,1,1),(g,1,2),(g,1,3),...,(g,1,9) (g,2,0),(g,2,1),(g,2,2),(g,2,3),\\
 &(g,2,4),...,(g,2,10),(g,3,0),(g,3,1),...,(g,3,11),..., (g,m,0),...,(g,m,m+8)\\
 &(g,-1,0),(g,-1,1),...,(g,-1,5),(g,-1,6)\}.
 \end{align*}
We now define a partial order on \( C_g \) by specifying the following relations:
 \begin{enumerate}
     \item $(g,-1)>(g,*_0)<(g,0)>(g,*_1)<(g,1)>(g,*_2)<(g,2)>(g,*_3)<\cdots >(g,*_m)<(g,m)$.
     \item $(g,i)>(g,i,0)<(g,i,1)>(g,i,2)<\cdots<(g,i,i+8)$; $(g,i,i+8)>(g,i,i+5)$ and $(g,i,i+6)>(g,i,i+1)$  if $i$ odd and $1\leq i\leq m$. 
     \item $(g,i)>(g,i,0)<(g,i,1)>(g,i,2)<\cdots>(g,i,i+8)$; $(g,i,i+5)>(g,i,i+8)$ and $(g,i,i+1)>(g,i,i+6)$ if $i$ even and $1\leq i\leq m$. 
     \item $(g,i)>(g,*_0)$ for every $1\leq i \leq m$.
     \item $(g,-1)>(g,-1,0)<(g,-1,1)>(g,-1,2)<...>(g,-1,6)$; $(g,-1)>(g,-1,4)$ and $(g,-1,6)<(g,-1,3)$.
 \end{enumerate}

It is evident that \( C_g \), equipped with the relations above, forms a partially ordered set of height 1. In Figure~\ref{fig_scheme_C_g_retract}, we illustrate the Hasse diagram of \( C_g \) for a finite group with a generating set consisting of two elements.

To conclude the construction, define
\[
X_r = \bigsqcup_{g\in G}^n C_{g},
\]
and extend the partial order by declaring:

\[
(g,i) > (h,*_0) \quad \text{whenever } h = g s_i \text{ for } i = 1, \ldots, m,
\]
while preserving the internal ordering within each \( C_g \). The resulting space \( X_G \) is a finite, connected partially ordered set of height 1.

\begin{figure}[ht]
    \centering
    \includegraphics[width=0.6\linewidth]{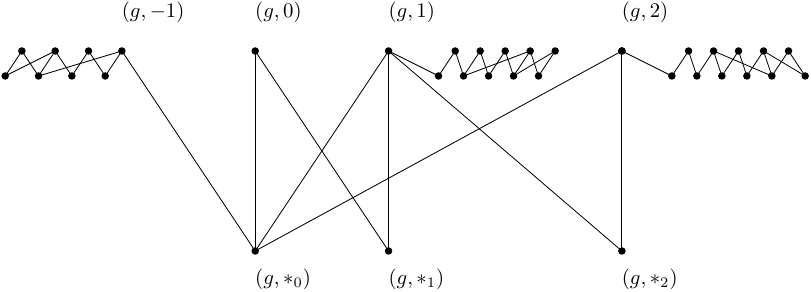}
    \caption{Hasse diagram of $C_g$ for a generating system of two elements.}
    \label{fig_scheme_C_g_retract}
\end{figure}

\begin{thm}\label{thm_retract}
Let \( r: G \rightarrow H \) be a group retraction, where \( G \) is a finite group of order \( n \). Then there exists a connected finite topological space \( X_r \) of height 1 and a topological retraction \( \overline{r}: X_r \rightarrow X_r \) that realizes \( r \) in both the topological category \( \mathbf{Top} \) and the homotopical category \( \mathbf{HTop} \).
\end{thm}
\begin{proof}
Let \( X_r \) be the finite topological space constructed as above. First, we prove that $\textnormal{Aut}(X_r)$ is isomorphic to $G$. Define a map \( T: G \to \textnormal{Aut}(X_r) \) by setting:
\[
T(g)(h,k) = (gh,k), \quad T(g)(h,*_i) = (gh,*_i), \quad T(g)(h,l,j) = (gh,l,j), \quad  T(g)(h,-1,t) = (gh,-1,t),
\]
for all \(-1\leq k\leq m\), \( 0 \leq i \leq m \),  \( 1 \leq l \leq m \), \( 0 \leq j \leq l+8 \) and \(0\leq t\leq 6\).

We first show that \( T \) is a well-defined homomorphism of groups. By the definition of $T$, the map \( T(g) \) clearly preserves the partial order defined on $X_r$, hence it is continuous. Its inverse is given by \( T(g^{-1}) \), so \( T(g) \in \textnormal{Aut}(X_r) \). Moreover, for any \( g, h \in G \), we have:
\[
T(g) \circ T(h) = T(gh),
\]
which shows that \( T \) is a homomorphism of groups.

To prove injectivity, suppose \( g \neq h \) and assume, for contradiction, that \( T(g) = T(h) \). Then:
\[
T(g)(e,0) = (g,0) = T(h)(e,0) = (h,0),
\]
which implies \( g = h \), a contradiction. Hence, \( T \) is injective.

We now prove surjectivity. Let \( f: X_r \to X_r \) be a homeomorphism. By construction, the only points \( x \in X_r \) such that \( |F_x| = 2m+3 \)  are precisely those of the form \( (h,*_0) \). Thus, we may assume:
\[
f((e,*_0)) = (h,*_0),
\]
for some \( h \in G \). We aim to show that \( f = T(h) \).

Let $x\in X_r$ such that $x>(g,*_0)$. Then $|U_x|=4$ if and only if  $x=(g,-1)$ or  $x=(g,m)$. On the other hand, there exists a unique $(g,*_0)\neq y<(g,-1)$ such that $|F_y|=4$ ($y=(g,-1,4)$), while this does not hold for $(g,m)$. By Proposition \ref{prop_fences} and this fact, we can conclude that $f((e,-1))=(h,-1)$. This also gives that $f((e,-1,4))=(h,-1,4)$ and, consequently, $f((e,-1,5))=(h,-1,5)$. From this, it is easy to very that $f((e,-1,i))=(h,-1,i)$ with $i=0,1,2,3,6$.

Observe that, in general, if \((g,-1) \neq x > (g,*_0) \), then either:
\begin{itemize}
    \item[(i)] \( x = (gs_i^{-1},i) \) for some generator \( s_i \in S \), or 
    \item[(ii)] \( x = (g,0) \).
\end{itemize}

In case (i), the point \( (gs_i^{-1},i) \) satisfies \( |U_{(gs_i^{-1},i)}| = 5 \) (if \( i \neq m \)) or \( 4 \) (if \( i = m \)). In case (ii), \( |U_{(g,0)}| = 3 \).  By Proposition~\ref{prop_fences}, the only possibility for $f$ is:
\[
f((e,0)) = (h,0),\quad f((e,m))=(h,m), \quad f((e,*_1)) = (h,*_1), \quad f((e,1)) = (h,1).
\]

Now consider the point \( (e,*_2) \), which is uniquely characterized by the relations:
\[
(e,1) > (e,*_2), \quad (e,2) > (e,*_2),
\]
and the fact that \( |U_{(e,2)}| = 5 \) if $2\neq m$ or \( 4 \) if $2=m$. The same structure holds for \( (h,*_2) \), so by Proposition~\ref{prop_fences}, we conclude:
\[
f((e,2)) = (h,2).
\]

Repeating this argument inductively, we obtain:
\[
f((e,*_i)) = (h,*_i), \quad f((e,i)) = (h,i), \quad \text{for } 3 \leq i \leq m,
\]
and consequently:
\[
f((e,l,j)) = (h,l,j), \quad \text{for all } 1 \leq l \leq m, \; 0 \leq j \leq l+8.
\]
Thus, \( f(C_e) = C_h \).

Now consider the point \( (s_i,*_0) \), which is the unique point satisfying:
\[
(s_i,*_0) < (e,i),
\]
and does not belong to \( C_e \). Since \( f \) is a homeomorphism, we must have:
\[
f((s_i,*_0)) = (hs_i,*_0),
\]
and repeating the previous arguments, we conclude that \( f \) and \( T(h) \) coincide on \( C_{s_i} \).

By iterating this argument and using the connectedness of \( X_G \), we conclude that \( f = T(h) \). Therefore, \( T \) is surjective. Hence, \( T: G \to \textnormal{Aut}(X_G) \) is an isomorphism of groups.

Let $\overline{r}:X_r\rightarrow X_r$ defined by 
\[
\overline{r}(g,k)=(r(g),k), \quad \overline{r}(g,*_i)=(r(g),*_i), \quad \overline{r}(g,l,j)=(r(g),l,j), \quad \overline{r}(g,-1,t)=(r(g),-1,t),
\]
for all  \( -1 \leq k \leq m \), \( 0 \leq i \leq m \), \( 1 \leq l \leq m \), \( 0 \leq j \leq l+8 \) and \(0\leq t\leq 6\). We prove that $\overline{r}$ is continuous. Consider $x,y\in X_r$ such that $x<y$ and distinguish cases: 
\begin{enumerate}
    \item $x,y\in C_g$ for some $g\in G$.
    \item $x\in C_g$ and $y\notin C_g$.
\end{enumerate}
If 1 holds, then $\overline{r}(x)<\overline{r}(y)$ by the definition of $\overline{r}$. Suppose now that 2 holds. Then $x=(g,*_0)$ and $y=(gs_i^{-1},i)$ for some $i=1,...,m$. By the definition of $\overline{r}$, we obtain 
\[ 
\overline{r}(x)=(r(g),*_0)), \quad \overline{r}(y)=(r(g)r(s_i^{-1}),i).
\]
Again, we distinguish cases. If $s_i\in \ker(r)$, then 
\[
\overline{r}(y)=(r(g)r(s_i^{-1}),i)=(r(g),i)>(r(g),*_0)=\overline{r}(x).
\]
If $s_i\notin \ker(r)$, then $s_i\in H$ and $s_i^{-1}\in H$. Therefore, 
\[
\overline{r}(y)=(r(g)r(s_i^{-1}),i)=(r(g)s_i^{-1},i)>(r(g),*_0)=\overline{r}(x).
\]
Hence, $\overline{r}$ is a continuous map. Moreover, it is evident that 
\[
\overline{r}(X_r)=\bigsqcup_{h\in i(H)} C_h,
\]
where we keep each partial order within $C_h$ and we have
\[
(g,i)>(h,*_0)\quad \textnormal{whenever } h=gr(s_i) \textnormal{ for } i=1,...,m.
\]
Note that for any $g\neq h$ in $i(H)$, $(g,i)>(h,*_0)$ if and only if  $s_i\in i(H)$ and $h=gs_i$. If $s_i\in \ker(r)$, then $(h,i)>(g,*_0)$ if and only if $g=h$.
It is also clear that $\overline{r}$ is the identity on $\overline{r}(X_r)$. Thus, $\overline{r}$ is a topological retraction. Additionally, by construction and adapting the proof to show that $\textnormal{Aut}(X_r)\simeq G$, we may conclude that $\textnormal{Aut}(\overline{r}(X_r))$ is isomorphic to $i(H)$ and also that $\overline{r}$ realizes $r$. To conclude, note that $X_r$ does not have beat points, which gives $\textnormal{Aut}(X_r)=\mathcal{E}(X_r)$ by Theorem \ref{thm_aut_igual_E}. From this, we conclude the result for the homotopical category $\mathbf{HTop}$.

\end{proof}

\begin{cor}\label{cor_realization}
    Let \( G \) be a finite group of order \( n \). Then there exists a connected finite topological space \( X_G \) of height 1 such that
\[
\mathrm{Aut}(X_G) \cong \mathcal{E}(X_G) \cong G.
\]
\end{cor}
\begin{proof}
Consider the identity retraction \( r: G \rightarrow G \). By Theorem~\ref{thm_retract}, there exists a connected finite topological space \( X_r \) of height 1 and a topological retraction \( \overline{r}: X_r \rightarrow X_r \) that realizes \( r \) in both \( \mathbf{Top} \) and \( \mathbf{HTop} \). Define \( X_G := X_r \). Then, by construction, we have
\[
\mathrm{Aut}(X_G) \cong \mathcal{E}(X_G) \cong G,
\]
as desired.
\end{proof}

Regarding the same question considered in Corollary~\ref{cor_realization}, but now for spaces of height \( 0 \), we may deduce the following result.
\begin{prop}\label{prop_height_zero}
Let \( X \) be a finite \( T_0 \)-topological space such that \( \mathrm{ht}(X) = 0 \). Then \( \mathrm{Aut}(X) \) or \( \mathcal{E}(X) \) is isomorphic to the symmetric group on \( |X| \) elements.
\end{prop}

\begin{proof}
If \( \mathrm{ht}(X) = 0 \), then \( X = \{x_1, \ldots, x_n\} \) is an antichain. In this case, the topology is discrete, and every bijection of \( X \) is a homeomorphism. Moreover, since \( X \) has no beat points, we have \( \mathrm{Aut}(X) = \mathcal{E}(X) \). Therefore, \( \mathrm{Aut}(X) = \mathcal{E}(X) \) is isomorphic to the symmetric group on \( |X| \) elements.
\end{proof}

Corollary~\ref{cor_realization} and Proposition~\ref{prop_height_zero} together show that the minimal height required to realize a finite group \( G \) as the group of automorphisms  or the group of self-homotopy equivalences of a finite topological space is 1, unless \( G \) is the symmetric group. In that unique case, \( G \) can be realized by a space of height 0. As stated in Section \ref{sec_intro}, the following question remains open: Given a finite group $G$ that is not a symmetric group, what is the minimal cardinality of a topological space $X$ of height 1 such that $\textnormal{Aut}(X)$ (or $\mathcal{E}(X)$) is isomorphic to $G$?

\bibliography{bibliografia}
\bibliographystyle{plain}

\newcommand{\Addresses}{{
  \bigskip
  \footnotesize

  \textsc{ P.J. Chocano, Departamento de Matemática Aplicada, Ciencia e Ingeniería de los Materiales y Tecnología Electrónica, ESCET Universidad Rey Juan Carlos, 28933 Móstoles (Madrid), Spain}\par\nopagebreak
  \textit{E-mail address}:\texttt{pedro.chocano@urjc.es}

}}

\Addresses
\end{document}